\documentclass{article}
\usepackage[utf8]{inputenc}
\usepackage{amscd, amssymb, amsfonts}
\usepackage{graphicx}
\usepackage[T1]{fontenc}
\usepackage{amsmath,amsthm}
\usepackage{amssymb}
\usepackage{cite}
\usepackage{kotex}
\usepackage{chngcntr}
\usepackage{lmodern}
\usepackage{ifthen}
\numberwithin{equation}{section}

\newtheorem{theorem}{Theorem}[section]
\newtheorem{lemma}[theorem]{Lemma}

\theoremstyle{definition}

\date{}
\theoremstyle{remark}
\newtheorem{remark}[theorem]{Remark}

\begin{document}

\allowdisplaybreaks
\title{\textbf{Global bounded solution of the chemotaxis attraction repulsion Cauchy problem with the nonlinear signal production in $\mathbb{R}^{N}$}}

\author{Tae Gab Ha$^1$ and Seyun Kim$^2$}

\maketitle

\centerline{Department of Mathematics, and Institute of Pure and Applied Mathematics,}
\centerline{Jeonbuk National University, Jeonju 54896, Republic of Korea}

\renewcommand{\thefootnote}{}
\footnote{Email: $^1$tgha@jbnu.ac.kr, $^2$dccddc8097@naver.com}

    \begin{abstract}
        In this paper, we consider the following attraction repulsion chemotaxis model with nonlinear signal term:
        \begin{align*}
            &u_{t}=\nabla \cdot(\nabla u-\xi_{1} u \nabla v +\xi_{2} u \nabla w), &\quad x \in \mathbb{R}^{N}, t>0, \\
            &0=\Delta v -\lambda_{1}v +f_{1}(u), \quad &\quad x \in \mathbb{R}^{N}, t>0, \\
            &0=\Delta w -\lambda_{2}w +f_{2}(u), \quad &\quad x \in \mathbb{R}^{N}, t>0,
        \end{align*}
where $\xi_{1},\xi_{2},\lambda_{1},\lambda_{2}$ are for some positive constants, and
\begin{equation*}
    f_{1} \in C^{1}([0,\infty)) \; \text{satisfying} \; 0 \leqslant f_{1}(s) \leqslant c_{1}s^{l},  \; \forall s \geqslant 0 \ \text{and} \ l> 0,
\end{equation*}
\begin{equation*}
    f_{2} \in C^{1}([0,\infty))\; \text{satisfying} \; 0 \leqslant f_{2}(s) \leqslant c_{2}s^{m},  \; \forall s \geqslant 0 \ \text{and} \ m> 0.
\end{equation*}
 We will show that this problem has a unique global bounded solution  when $ l>\frac{2}{N}, l<m \ \text{with} \ m \geqslant 1$, or $l=m<\frac{2}{N}$.
    \end{abstract}
\section{Introduction}

The classical chemotaxis system
 \begin{align*}
    &u_{t}=\nabla \cdot(\nabla u-\xi( u,v) \nabla v ), &\quad x \in \Omega, t>0,   \\
    &\tau v_{t}=\Delta v -\lambda v +f(u), \quad &\quad x \in \Omega, t>0
\end{align*}
was proposed by Keller and Segel \cite{keller} in 1970  with $f(u)=u$ to describe the aggregation of cellular slime mold.
In this model,\ $u=u(x,t)$ denotes the cell density and $v=v(x,t)$ denotes the concentration of the chemoattractant. The function $\xi(u,v)$ represents the sensitivity with respect to chemotaxis, and the function $f(u)$ models the growth of the chemoattractant.

In recent decades, significant progress has been made in analyzing various cases of the chemotaxis system both on bounded and unbounded domains. For example, when $f(u)=\alpha u$ , $\tau=0$, $\xi(u,v)=u$ and $\Omega=\mathbb{R}^{N}$, with $N=2,$ if the initial data $u_{0}$ satisfies that $\alpha  \int_{\mathbb{R}^{2}} u_{0} dx <8\pi$, then the solution exists globally. In the contrary if $\alpha  \int_{\mathbb{R}^{2}} u_{0} dx >8\pi$, then the solution blows up in finite time. Moreover when $N\geqslant3$, if the moment $\int_{\mathbb{R}^{N}} u_{0}(x)\vert x-q\vert^{2} dx$ is sufficiently small then the solution blows up in finite time (\cite{nagai2,nagai3}). On the other hand, when $\lambda=1,f(u)=u,\xi(u,v)=u$ and $N\geqslant2$, it has a global weak solution, and has polynomial decay property for sufficiently small initial data (\cite{suki2}).
When $\Omega \subset \mathbb{R}^{N}$ with smooth bounded domain and homogeneous Neumann  boundary condition, a similar situation occurs, see \cite{nagai9,nagai10} and a list on references therein.

Above mentioned references considered the existence of solutions which are globally defined in time or blow up at a finite time and the asymptotic behavior of the global solutions. In addition to these studies, the global boundedness problem is also being studied  in various ways. For instance, when $f(u)=\alpha u$ , $\tau=1$, $\xi(u,v)=u,\lambda=1$ and $\Omega=\mathbb{R}^{N}$ with $N\geqslant2 $, there is a unique bounded classical solution which decays to zero as $t\rightarrow \infty$ and behaves like the heat kernel (\cite{nagai5,nagai6}). On the other hand, when $\Omega \subset \mathbb{R}^{2}$ is a smooth bounded domain the solution blows up in finite time for some large initial data (\cite{blow2}), while for small initial data satisfying $\int_{\Omega} u_{0} dx <\frac{8\pi}{\alpha}$, the solution exists globally (\cite{nagai11}). When $N\geqslant 3$, the solution of the classical chemotaxis system may occur the blow-up phenomenon (\cite{winkler3}).

Above mentioned references considered the linear production term. However, the problem considering nonlinear production term is much more complicated. When $0\leqslant f(u) \leqslant cu^{k}$ with $0<k<\frac{2}{N}$, for any higher dimensions without any small initial data assumption the global solution is bounded if $\Omega$ is a bounded domain (\cite{dliu2}). If the second equation of the classical chemotaxis system is replaced by $0=\Delta v- \frac{1}{\vert \Omega \vert}\int_{\Omega}f(u(x,t))dx+f(u)$ with $f(u)\geqslant u^{k}$ for $u\geqslant 1$ and $k>\frac{2}{N}$, Winkler \cite{winkler2} show that the radially symmetric solution blows up in finite time. So $k=\frac{2}{N}$ is the critical number in the nonlinear production point of view.  But when the sensitive term $ \xi(u,v)=\frac{u}{v}$, Liu \cite{dliu1} proved that the solution exists globally without any assumption of small initial data and restriction of high dimensions for $0\leqslant f(u) \leqslant cu^{k}$, $0<k<\frac{2}{N}+2$ if $\Omega$ is a bounded domain. Moreover, Frassu and Viglialoro \cite{gis2} proved that there exists global bounded solution when $\xi(u)\leqslant c_{1}u^{\alpha}, 0\leqslant f(u) \leqslant c_{2}u^{l} , \frac{2}{N} \leqslant \alpha < 1+\frac{1}{N}-\frac{l}{2}$ and $l<\frac{2}{N}$. Here the signal production $f$ is subcritical $l<\frac{2}{N}$. We know that the signal production term of the classical chemotaxis model, $f(u)=u$, is critical when $N=2$. But as far as we know that for the critical case ($l=\frac{2}{N}$) of the general chemotaxis models, the global existence or blow-up of the solutions is still open problem both when the domain $\Omega$ is bounded or unbounded.

Luca et al. \cite{luca} proposed the following model:
\begin{align*}
    &u_{t}=\nabla \cdot(\nabla u-\xi_{1} u \nabla v +\xi_{2} u \nabla w), &\quad x \in \Omega, t>0, \nonumber  \\
    &\tau_{1}v_{t}=\Delta v -\lambda_{1}v +f_{1}(u), \quad &\quad x \in \Omega, t>0, \\
    &\tau_{2}w_{t}=\Delta w -\lambda_{2}w +f_{2}(u), \quad &\quad x \in \Omega, t>0,  \nonumber
\end{align*}
where $u(x,t)$ denotes the density of cells, $v(x,t)$ denotes the chemical concentration of attractant and $w(x,t)$ denotes the chemical concentration of repellent. When $\Omega=\mathbb{R}^{N},f_{1}(u)=f_{2}(u)=u,\tau_{1}=\tau_{2}=0$, for the repulsion dominant case, i.e. $\xi_{1}<\xi_{2}$, the solution is global bounded and decays to zero polynomial speed as $t \rightarrow \infty$ without any assumption of size of the initial data, on the other hand, for the attraction dominated case, i.e. $\xi_{1}>\xi_{2} $, the solution blows up in finite time when $\Vert u_{0} \Vert_{1} >\frac{8\pi}{\xi_{1}-\xi_{2}}$ and $\int_{\mathbb{R}^{2}}u_{0}\vert x-q\vert ^{2} dx$ is sufficiently small (\cite{regular2}). When $\Omega=\mathbb{R}^{N},f_{1}(u)=c_{1}u,f_{2}(u)=c_{2}u,\tau_{1}=\tau_{2}=1,$ and $\xi_{1}\lambda_{1}=\xi_{2}\lambda_{2}$, the solution is global bounded and polynomially decays to zero as $t\rightarrow \infty$ (\cite{jin}). When $N=2$ , $\xi_{1}\geqslant \xi_{2}$ and $\tau_{1}=\tau_{2}=0$, If $\int_{\mathbb{R}^{2}} u_{0} dx <\frac{8\pi}{\xi_{1}-\xi_{2}}$, Nagai and Yamada showed that solution exists globally (\cite{nagai7}). They estimated the entropy $\Vert (1+u)\log(1+u)\Vert_{1}\leqslant C(T)$ and then proved $\Vert u \Vert_{\infty}<C(T)$, where $C(T)$ is finite. For $\tau_{1}=1,\tau_{2}=0$ or $\tau_{1}=\tau_{2}=1$, global boundedness of the solution was obtained by using the technique of \cite{nagai7} (see \cite{shi2,nagai1}). On the other hand, when $\Omega \subset \mathbb{R}^{N}$ is a smooth bounded domain with homogeneous Neumann  boundary condition, if the second and third equations of the model of Luca et al. \cite{luca} above are replaced by $0=\Delta v -\frac{1}{\vert \Omega \vert}\int_{\Omega}f_{1}(u)dx+f_{1}(u)$ and $0=\Delta w -\frac{1}{\vert \Omega \vert}\int_{\Omega}f_{2}(u) dx+f_{2}(u)$, respectively, the solution blows up in finite time when $f_{1}(u) \geqslant c_{1}u^{k_{1}}, f_{2}(u) \leqslant c_{2}u^{k_{2}} \ \forall u\geqslant 1$ and $k_{1}>k_{2}, k_{1}>\frac{2}{N}$ (\cite{blow3}). Recently, Columbu et al. proved that if $0 \leqslant f_{1}(u) \leqslant c_{1}u^{k_{1}}, c_{2}(1+u)^{k_{2}} \leqslant  f_{2}(u) \leqslant c_{2}(1+u)^{k_{2}}$, then the solution exists globally when
$k_{2}>k_{1},or \ 0 <k,l<\frac{2}{N}$. They also studied boundedness criteria when $\tau_{1}=\tau_{2}=1$, and the second and third equations of the model of Luca et al. \cite{luca} above are replaced by  $0=\Delta v -\frac{1}{\vert \Omega \vert}\int_{\Omega}f_{1}(u)dx+f_{1}(u)$ and $0=\Delta w -\frac{1}{\vert \Omega \vert}\int_{\Omega}f_{2}(u) dx+f_{2}(u)$, respectively (\cite{gis}). Also Viglialoro studied the same problem on bounded domain of attraction-repulsion nonlinear signal case (\cite{gis3}). But for $\Omega=\mathbb{R}^{N}$, it is difficult to assume that  $c_{2}(1+u)^{k_{2}} \leqslant  f_{2}(u) \leqslant c_{2}(1+u)^{k_{2}}$, furthermore there is none, as far as we know, for the nonlinear signal case.

Motivated above mentioned references, we will study the global bounded solution of the following attraction repulsion problem with the nonlinear signal production in $\mathbb{R}^n$:
\begin{align}\label{0}
    &u_{t}=\nabla \cdot(\nabla u-\xi_{1} u \nabla v +\xi_{2} u \nabla w), &\quad x \in \mathbb{R}^{N}, t>0, \nonumber  \\
    &0=\Delta v -\lambda_{1}v +f_{1}(u), \quad &\quad x \in \mathbb{R}^{N}, t>0, \\
    &0=\Delta w -\lambda_{2}w +f_{2}(u), \quad &\quad x \in \mathbb{R}^{N}, t>0.  \nonumber
\end{align}

where $\xi_{1},\xi_{2},\lambda_{1},\lambda_{2}$ are positive constants.
We also assume that
\begin{equation*}
    f_{1} \in C^{1}([0,\infty)) \; \text{satisfying} \; 0 \leqslant f_{1}(s) \leqslant c_{1}s^{l},  \;   l> 0 \;  \text{and} \quad \forall s \geqslant 0.
\end{equation*}
\begin{equation*}
    f_{2} \in C^{1}([0,\infty))\; \text{satisfying} \; 0 \leqslant f_{2}(s) \leqslant c_{2}s^{m},  \;   m> 0 \;  \text{and} \quad \forall s \geqslant 0.
\end{equation*}
 We can extend $f_{i}(s)$ to the negative axis by defining $f_{i}(s)=sf_{i}'(0)$ for $s<0,i=1,2$ as in \cite{winkler1}.
 This implies $f_{i} \in C^{1}(\mathbb{R}).$\\
 \begin{theorem}\label{thm}
        For $N \geqslant 2 $, assume that $0\leqslant f_{1}(s) \leqslant c_{1}s^{l},f_{2}(s)=s^{m}$, where $l > \frac{2}{N},m \geqslant 1$ and $l <m$. Then the solution of the problem (1.1) has a global bounded solution such that
        \begin{align*}
            &\Vert u \Vert_{p} \leqslant C_{1} \quad \text{for} \ 1\leqslant p \leqslant \infty, \\
            &\Vert v \Vert_{p} \leqslant C_{2} \quad \text{for} \ \frac{1}{l} \leqslant p \leqslant \infty ,\ \text{when} \ l<1, \\
            &\Vert v \Vert_{p} \leqslant C_{3} \quad \text{for} \ 1\leqslant p \leqslant \infty ,\ \text{when} \ l\geqslant 1,\\
            &\Vert w \Vert_{p} \leqslant C_{4} \quad \text{for} \ 1\leqslant p \leqslant \infty.
        \end{align*}
    In addition, if $0\leqslant f_{1}(s) \leqslant c_{1}s^{l},0\leqslant f_{2}(s) \leqslant c_{2}s^{m}\ \text{for} \ l=m <\frac{2}{N}$, then it holds that
    \begin{align*}
        &\Vert u \Vert_{p}\leqslant C_{5} \quad \text{for} \ 1\leqslant p \leqslant \infty, \\
        &\Vert (v,w) \Vert_{p} \leqslant C_{6} \quad \text{for} \ \frac{1}{l}=\frac{1}{m} \leqslant p \leqslant \infty.
    \end{align*}
    Positive constants \ $C_{1},C_{2},C_{3},C_{4},C_{5},C_{6}$ depend on the $l,m,\lambda_{1},\lambda_{2},\Vert u_{0} \Vert_{1},\\ \Vert u_{0} \Vert_{\infty},c_{1},c_{2},\xi_{1},\xi_{2},N$.
\end{theorem}

We will prove the existence of the mild solution by using the Banach fixed point method and the semigroup property, and then prove the H\"older regularity of the mild solution. Next we do linearization of the parabolic equation, and then we get the classical solution applying the methods developed in \cite{regular1} and \cite{regular3} (Lemma \ref{lem6}, \ref{lem7}, \ref{lem8}). In Lemma \ref{lem9} and \ref{lem10}, we estimate the boundedness of the solution $u, v, w$, and then by using the semigroup property, we obtain the global uniformly bounded solution.

\section{Preliminary}

Let $X=L^{p}(\mathbb{R}^{N})$ with $p \geqslant 1$ or $X=BUC(\mathbb{R}^{N})$, where $BUC(\mathbb{R}^{N})$ is a bounded uniformly continuous Banach space with the norm $\Vert \Vert_{\infty}$. $A=-\Delta +\lambda I$ is the Sectorial operator generated by the analytic semigroup $T(t)_{t\geqslant 0}$ with $Re(\sigma(A))>0, \lambda>0$ in $X$ (see \cite{hen,pazy}).
This induces the inverse operator
\begin{equation*}
    (\lambda I -\Delta)^{-1} u(x) =\int_{0}^{\infty} e^{-\lambda s} G(x,s)\ast u(x) ds,
\end{equation*}
where $ u \in L^{1}(\mathbb{R}^{N})$ or $ \in BUC(\mathbb{R}^{N})$, and $G(x,t)=(4 \pi t)^{-\frac{N}{2}}e^{-\frac{\vert x \vert^{2}}{4t}}$. $ \Gamma_{\lambda}:=\int_{0}^{\infty} e^{-\lambda s} G(x,s) ds$ is called the
Bessel potential (see \cite{hen,pazy,stein}). Let $C^{v}(\mathbb{R}^{N})$ be the H\"older space with norm
\begin{equation*}
    \Vert u \Vert_{C^{v}}:= \Vert u \Vert_{\infty}+ \sup_{x,y \in \mathbb{R}^{N} \ \text{and} \ x \ne y} \frac{u(x)-u(y)}{\vert x-y \vert ^{v}} .
\end{equation*}
We introduce the parabolic fundamental solutions (cf. \cite[Ch IV]{rus}, \cite[Ch 1]{fri}). Consider the following parabolic equation:
\begin{equation}\label{1}
 Lu\equiv \sum_{i,j}^{N}a_{ij}(x,t)\frac{\partial^{2} u}{\partial x_i x_j}+\sum_{i}^{N}b_{i}(x,t)\frac{\partial u}{\partial x_i}+c(x,t)u-\frac{\partial u}{\partial t}=0, \; (x,t) \in \mathbb{R}^{N}\times (0,\infty).
\end{equation}
If $b_{i}(x,t),c(x,t)=0$ ,then it is well known that there is a fundamental solution the following form:
\begin{align*}
    Z_{0}(x-\xi,\xi,&t,\tau) \\ = & \frac{1}{\left[4\pi(t-\tau)\right]^{\frac{N}{2}}(detA(\xi,\tau))^{\frac{1}{2}}} e^{\left(-\frac{1}{4(t-\tau)}\sum_{i,j}^{N}A^{ij}(\xi,\tau)(x_{i}-\xi_{i})(x_{j}-\xi_{j})\right)},
\end{align*}
where $A_{ij}$ is the matrix composed of the leading coefficients $a_{i,j}$ of the (\ref{1}), while the $A^{i,j}$ are the elements of the inverse matrix of $A_{ij}$. For $t<\tau$ we set $Z_{0}=0$. If $b_{i},c \ne 0$, then we construct fundamental solution $Z$ the following form:
\begin{equation}\label{2}
    Z(x,\xi,t,\tau)=Z_{0}(x-\xi,\xi,t,\tau)+\int_{\tau}^{t} d\lambda \int_{\mathbb{R}^{N}}Z_{0}(x-y,y,t,\lambda)\Phi(y,\xi,\lambda,\tau)dy,
\end{equation}
where $\Phi$ is the solution of the Volterra integral equation (\cite[p14]{fri}). There is an useful lemma for our paper.
\begin{lemma}[\cite{rus} p376-377, \cite{fri} p23]\label{lem1}
    Assume that $L$ is uniformly parabolic in $\mathbb{R}^{N} \times [0,T]$ and $a_{i,j},b_{i},c \in C^{\alpha,\frac{\alpha}{2}}(\mathbb{R}^{N}\times [0,T])$ with $ 0< \alpha <1$.
    Then the problem (\ref{1}) has the fundamental solution (\ref{2}). In addition the fundamental solution $Z$ has following estimates
    \begin{equation*}
        \left\vert D_{t}^{r}D_{x}^{s}Z(x,\xi,t,\tau)\right\vert \leqslant c(t-\tau)^{-\frac{N+2r+s}{2}}exp\left(-C\frac{\left\vert x-\xi\right\vert ^{2}}{t-\tau}\right),
    \end{equation*}
    where $2r+s\leqslant2,t>t'>\tau$,
    \begin{align*}
        &\left\vert D_{t}^{r}D_{x}^{s}Z(x,\xi,t,\tau)-D_{t}^{r}D_{x'}^{s}Z(x',\xi,t,\tau)\right\vert \\ \leqslant &c\left[\vert x-x'\vert^{\gamma}(t-\tau)^{-\frac{N+2+\gamma}{2}}+\vert x-x'\vert^{\beta}(t-\tau)^{-\frac{N+2-\alpha+\beta}{2}}\right] exp\left(-C\frac{\left\vert x''-\xi\right\vert ^{2}}{t-\tau}\right),\nonumber
    \end{align*}
    where $2r+s=2$,$ 0\leqslant \gamma \leqslant 1, 0\leqslant \beta \leqslant \alpha, t>\tau,$ and
    \begin{align*}
        &\left\vert D_{t}^{r}D_{x}^{s}Z(x,\xi,t,\tau)-D_{t}^{r}D_{x}^{s}Z(x,\xi,t',\tau)\right\vert \\ \leqslant &c\left[(t-t')(t'-\tau)^{-\frac{N+2r+s+2}{2}}+(t-t')^{\frac{2-2r-s+\alpha}{2}}(t'-\tau)^{-\frac{N+2}{2}}\right] exp\left(-C\frac{\left\vert x-\xi\right\vert ^{2}}{t-\tau}\right),\nonumber
    \end{align*}
    where $2r+s=1,2$ and $t>t'>\tau$.
\end{lemma}
We also introduce useful lemmas which will play an essential role when establishing the our result.
\begin{lemma}[\cite{hen} Theorem 1.4.3]\label{lem2}
    Suppose that $A$ is sectorial and $Re\sigma(A)>\delta>0$. Let the analytic semigroup ${T(t)}_{t\geqslant 0}$ be generated by $A$.  For $\alpha\geqslant0$, there exists $C_{\alpha} <\infty$ such that
    \begin{equation*}
        \Vert A^{\alpha}T(t) \Vert \leqslant C_{\alpha}t^{-\alpha}e^{-\delta t} \quad \text{for} \quad t>0,
    \end{equation*}
and if $0<\alpha \leqslant 1 ,x \in D(A^{\alpha})$,
    \begin{equation*}
        \Vert (T(t)-I)x\Vert \leqslant \frac{1}{\alpha} C_{1-\alpha}t^{\alpha}\Vert A^{\alpha}x\Vert.
    \end{equation*}\qed
\end{lemma}

\begin{lemma}[\cite{regular1} Lemma 3.1]\label{lem3}
    Let $p \in [1,\infty)$ and ${T(t)}_{t\geqslant0}$ be the semigroup generated by $\Delta-I$ on $L^{p}(\mathbb{R}^{N})$. For every $t>0$, the operator $T(t)\nabla \cdot $ has a unique bounded extension on $(L^{p}(\mathbb{R}^{N}))^{N}$ satisfying
    \begin{equation*}
        \Vert T(t)\nabla \cdot u\Vert_{p} \leqslant C_{1} t^{-\frac{1}{2}}e^{-t}\Vert u \Vert_{p} \quad \forall u \in (L^{p}(\mathbb{R}^{N}))^{N}, t>0,
    \end{equation*}
    where $C_{1}$ depends only on $p$ and $N$. Furthermore, for every $q\in [p,\infty]$ we have that $T(t)\nabla \cdot u \in L^{q}(\mathbb{R}^{N})$ with
    \begin{equation*}
        \Vert T(t)\nabla \cdot u\Vert_{q} \leqslant C_{2} t^{-\frac{1}{2}-\frac{N}{2}(\frac{1}{p}-\frac{1}{q})}e^{-t}\Vert u \Vert_{p} \quad \forall u \in (L^{p}(\mathbb{R}^{N}))^{N}, t>0,
    \end{equation*}
    where $C_{2}$ is constant depending only on $N,q$ and $p$.
\end{lemma}
\begin{lemma}[\cite{regular3} Lemma 2.2]\label{lem4}
For every $u \in BUC(\mathbb{R}^{N})$, we have that for  $\lambda >0$,
\begin{equation*}
    \Vert (\Delta-\lambda)^{-1}u \Vert_{\infty} \leqslant \frac{1}{\lambda}\Vert u \Vert_{\infty}
\end{equation*}
 and
\begin{equation*}
    \Vert \nabla (\Delta - \lambda)^{-1} u \Vert_{\infty} \leqslant \frac{\sqrt{N}}{\sqrt{\lambda}} \Vert u \Vert_{\infty}.
\end{equation*}

\end{lemma}
Let $W^{k,p}(\mathbb{R}^{N})$ be the space of functions $ u $ in $L^{p}(\mathbb{R}^{N})$ whose distribution derivatives of order up to $k$ are also in $L^{p}(\mathbb{R}^{N})$ (\cite{nonlinear}).\ Let we denote $W^{k,2}(\mathbb{R}^{N})=H^{k}(\mathbb{R}^{N})$.
\begin{lemma}[\cite{hen} p77]\label{lem5}
Let $X^{0}=L^{p}(\mathbb{R}^{N})$ and $ X^{1}=Dom(\Delta-I)=W^{2,p}(\mathbb{R}^{N})$. Then we have the following property for $0\leqslant \alpha \leqslant 1 $ and $2\leqslant p < \infty $ with $p>\frac{N}{2}$,
\begin{align*}
    &X^{\alpha}=W^{k,p}(\mathbb{R}^{N}), \quad \text{when} \ 2\alpha=k \ \text{ is an integer,}\\
    &X^{\alpha} \subset L^{q}(\mathbb{R}^{N}), \quad \text{when} \ 2p\alpha <N, \ \frac{1}{q} \geqslant \frac{1}{p}-\frac{2\alpha}{N} , \ \infty>q\geqslant p,\\
    &X^{\alpha} \subset C^{v}(\mathbb{R}^{N}), \quad \text{when} \ 2p\alpha>N, \ 0 \leqslant v < 2\alpha-\frac{N}{p},\\
    &X^{\alpha} \subset W^{k,q}(\mathbb{R}^{N}), \quad \text{when} \ 0 \leqslant k \leqslant 2\alpha \ \text{and}\ \frac{1}{q} \geqslant \frac{1}{p}-\frac{2\alpha-k}{N}.
\end{align*}
\end{lemma}
\begin{lemma}[\cite{suki1} Lemma 2.4]\label{lem11}
 Let $N\geqslant1,m\geqslant1,a>2,u \in L^{q_{1}}(\mathbb{R}^{N})$ with $q_{1}\geqslant 1$ and $u^{\frac{r+m-1}{2}} \in H^{1}(\mathbb{R}^{N})$ with $r>0$. If $q_{1} \in [1,r+m-1],q_{2} \in \left[\frac{r+m-1}{2},\frac{a(r+m-1)}{2}\right]$ and
    \begin{align*}
        &1\leqslant q_{1} \leqslant q_{2} \leqslant \infty \quad \text{when} \ N=1, \nonumber \\
        &1\leqslant q_{1} \leqslant q_{2} < \infty \quad \text{when} \ N=2,  \\
        &1\leqslant q_{1} \leqslant q_{2} \leqslant  \frac{(r+m-1)N}{N-2} \quad \text{when} \ N\geqslant 3,\nonumber
    \end{align*}
    then it holds that
    \begin{align*}
        &\Vert u\Vert_{q_{2}} \leqslant C^{\frac{2}{r+m-1}}\Vert u \Vert_{q_{1}}^{1-\Theta}\Vert \nabla u^{\frac{r+m-1}{2}}\Vert_{2}^{\frac{2\Theta}{r+m-1}} \text{with} \\
        &\Theta=\frac{r+m-1}{2}\left(\frac{1}{q_{1}}-\frac{1}{q_{2}}\right)\left(\frac{1}{N}-\frac{1}{2}+\frac{r+m-1}{2q_{1}}\right)^{-1} ,
    \end{align*}
    where
    \begin{align*}
        &C \ \text{depends only on}\ N \ \text{and} \ \alpha \quad \text{when} \quad q_{1}\geqslant \frac{r+m-1}{2},\\
        &C=c_{0}^{\frac{1}{\beta}} \ \text{with} \ c_{0} \  \text{depending only on} \ N \ \text{and}\ \alpha \quad \text{when}\ 1\leqslant q_{1} <\frac{r+m-1}{2},
    \end{align*}
    and
    \begin{equation*}
        \beta=\frac{q_{2}-\frac{r+m-1}{2}}{q_{2}-q_{1}}\left[\frac{2q_{1}}{r+m-1}+\left(1-\frac{2q_{1}}{r+m-1}\right)\frac{2N}{N+2}\right].
    \end{equation*}
    \qed
\end{lemma}

\section{Proof of Theorem 1.1}

This section is devoted to prove Theorem 1.1. At first, we will prove the existence of the mild solution by using the Banach fixed point method and the semigroup property.

\begin{lemma}\label{lem6}
    Assume that the nonnegative initial data, $u_{0}$, of the problem (\ref{0}) belongs to $L^{1}(\mathbb{R}^{N})\cap BUC(\mathbb{R}^{N})$ with $N\geqslant2$, then (1.1) has a unique local mild solution such that $ u \in C([0,T];L^{1}(\mathbb{R}^{N})\cap BUC(\mathbb{R}^{N}))$ and $ u(t,\cdot) \in X^{\beta}$ for $0<\beta <\frac{1}{2}$ where $X^{0}=L^{p}(\mathbb{R}^{N}),X^{1}=W^{2,p}(\mathbb{R}^{N})$ with $p>N$ or $X^{0}=BUC(\mathbb{R}^{N})$.
\end{lemma}
\begin{proof}
    Put $R > \Vert u_{0} \Vert_{1}+\Vert u_{0} \Vert_{\infty}$
    \ and
    \begin{equation*}
     S_{R,T}:=\{u\in C([0,T];L^{1}(\mathbb{R}^{N})\cap BUC(\mathbb{R}^{N})) : \sup_{0\leqslant t \leqslant T}(\Vert u \Vert_{1} +\Vert u \Vert_{\infty}) \leqslant R\} ,
    \end{equation*}
    with the norm $\Vert \cdot \Vert _{S_{R,T}}:= \sup_{0\leqslant t \leqslant T} \Vert u \Vert_{X}$ where $\Vert \cdot \Vert_{X} =\Vert \cdot \Vert_{1} + \Vert \cdot \Vert_{\infty}$.
    \begin{align}
        Ku(t):=T(t)u_{0}+\int_{0}^{t}T(t-s) \nabla \cdot (\xi_{1}u(s)\nabla(\Delta-\lambda_{1}I)^{-1}f_{1}(u(s))-\nonumber \\ \xi_{2}u(s)\nabla(\Delta-\lambda_{2}I)^{-1}f_{2}(u(s))) ds+\int_{0}^{t}T(t-s)u(s)ds.\nonumber
    \end{align}
    Put
    \begin{align*}
        &F_{1}(s):=T(t-s) \nabla \cdot (\xi_{1}u(s) \nabla(\Delta-\lambda_{1}I)^{-1}f_{1}(u(s))-\xi_{2}u(s)\nabla(\Delta-\lambda_{2}I)^{-1}f_{2}(u(s))),\\
        &F_{2}(s):=T(t-s)u(s) \; \text{for}\ s \in [0,T].
    \end{align*}
    \textbf{Claim 1}: $(Ku)(t)$ is well defined for $u \in S_{R,T}$ and $t \in [0,T].$

    Put
    \begin{align*}
        &F_{1,\epsilon}(s)\\
        &:=T(t-s) \nabla \cdot (\xi_{1}u(s) \nabla(\Delta-\lambda_{1}I)^{-1}f_{1}(u(s))-\xi_{2}u(s)\nabla(\Delta-\lambda_{2}I)^{-1}f_{2}(u(s)))\\
        &=T(t-\epsilon-s)T(\epsilon)\nabla \cdot (\xi_{1}u(s) \nabla(\Delta-\lambda_{1}I)^{-1}f_{1}(u(s))-\xi_{2}u(s)\nabla(\Delta-\lambda_{2}I)^{-1}f_{2}(u(s))),
    \end{align*}
    for $0<\epsilon <t $ and $s \in [0,t-\epsilon].$ By using Lemma \ref{lem3} and \ref{lem4} we have that
    \begin{equation*}
        [0,t-\epsilon] \ni s \mapsto T(\epsilon)\nabla \cdot (\xi_{1}u(s) \nabla(\Delta-\lambda_{1}I)^{-1}f_{1}(u(s))-\xi_{2}u(s)\nabla(\Delta-\lambda_{2}I)^{-1}f_{2}(u(s)))
    \end{equation*}
    is continuous for every $0<\epsilon<t\in[0,T]$. Since $\epsilon>0$ is arbitrary, the function $F_{1}:[0,T) \rightarrow L^{1}(\mathbb{R}^{N}) \cap BUC(\mathbb{R}^{N})$ is continuous. Moreover, using Lemma \ref{lem3} and \ref{lem4} we obtain
    \begin{align*}
        \int_{0}^{t} \Vert F_{1}(s) \Vert_{X} &\leqslant c \int_{0}^{t} \Vert T(t-s) \nabla \cdot \Vert_{X}\Vert u(s) \Vert_{X}(\Vert \nabla(\Delta-\lambda_{1}I)^{-1}f_{1}(u(s)) \Vert_{\infty}\\ & \quad + \Vert \nabla(\Delta-\lambda_{2}I)^{-1}f_{2}(u(s))\Vert_{\infty} ) ds\\
        &\leqslant c(\xi_{1},\xi_{2},R,l,m,\lambda_{1},\lambda_{2},N) \int_{0}^{t} (t-s)^{-\frac{1}{2}}e^{-(t-s)} ds \\ &\leqslant c(\xi_{1},\xi_{2},R,l,m,\lambda_{1},\lambda_{2},N) \Gamma(\frac{1}{2}),
    \end{align*}
    where  $\Gamma(\cdot)$ is the Gamma function. Thus, the integral $\int_{0}^{t} F_{1}(s) ds$ in $L^{1}(\mathbb{R}^{N})\cap BUC(\mathbb{R}^{N})$.

    By the similar argument, we conclude that
    \begin{equation*}
        [0,T) \ni s \mapsto F_{2}(s):=T(t-s)u(s) \in L^{1}(\mathbb{R}^{N}) \cap BUC(\mathbb{R}^{N})
    \end{equation*}
    is continuous and $\int_{0}^{t} F_{2}(s) ds \in L^{1}(\mathbb{R}^{N}) \cap BUC(\mathbb{R}^{N}).$

    \textbf{Claim 2}: There exists $T=T(R)$ such that $K$ maps $S_{R,T}$ into itself.

    Let $u\in S_{R,T}$. Then we have
    \begin{equation}
    \begin{aligned}
        \Vert Ku \Vert_{1} \leqslant &\Vert T(t)u_{0} \Vert_{1}\nonumber\\ & \ + \int_{0}^{t} \Vert T(t-s)\nabla \cdot (\xi_{1}u(s) \nabla (\lambda_{1}I-\Delta)^{-1} f_{1}(u(s))-\\ & \quad \xi_{2}u(s) \nabla (\lambda_{1}I-\Delta)^{-1}f_{1}(u(s))) \Vert_{1} ds \nonumber \\
        & \ + \int_{0}^{t}\Vert T(t-s)u(s)\Vert_{1}ds \\
        \leqslant & e^{-t}\Vert u_{0}\Vert_{1}+C(R,l,m,\lambda_{1},\lambda_{2},N,\xi_{1},\xi_{2})\int_{0}^{t}(t-s)^{-\frac{1}{2}}e^{-(t-s)}ds\nonumber\\
        &+\int_{0}^{t}e^{-(t-s)}\Vert u(s)\Vert_{1}ds\\
        \leqslant & e^{-t}\Vert u_{0} \Vert_{1}+C(R,l,m,N,\xi_{1},\xi_{2})t^{\frac{1}{2}}+(1-e^{-t})R\\
        \leqslant &R.
    \end{aligned}
    \end{equation}
    for sufficiently small $t$. By the similar argument, $\Vert Gu(t) \Vert_{\infty} \leqslant R$ for sufficiently small $t$. Thus $K(S_{R,T}) \subseteq S_{R,T}$ for sufficiently small $t$.

   \textbf{Claim 3}: $K$ is a contraction mapping.

   Let $u_{1},u_{2} \in S_{R,T}$. Then we obtain
   \begin{align*}
    &\Vert Ku_{1}(t)-Ku_{2}(t)\Vert_{1}  \\
    &\leqslant \int_{0}^{t} \Vert T(t-s)\nabla \cdot (\xi_{1}u_{1}(s)\nabla (\lambda_{1}I-\Delta)^{-1} f_{1}(u_{1}(s))-\xi_{1}u_{2}(s)\nabla (\lambda_{1}I-\Delta)^{-1} f_{1}(u_{2}(s)))\Vert_{1}ds \\
    & \ + \int_{0}^{t} \Vert T(t-s)\nabla \cdot (\xi_{2}u_{1}(s)\nabla (\lambda_{2}I-\Delta)^{-1} f_{2}(u_{1}(s))-\xi_{2}u_{2}(s)\nabla (\lambda_{2}I-\Delta)^{-1} f_{2}(u_{2}(s)))\Vert_{1}ds \\
    & \ + \int_{0}^{t}\Vert T(t-s)\left(u_{1}(s)-u_{2}(s)\right)\Vert_{1} ds\\
    &\leqslant c(N,\xi_{1}) \int_{0}^{t} (t-s)^{-\frac{1}{2}}e^{-(t-s)}\Vert u_{1}(s)-u_{2}(s) \Vert_{1} \Vert \nabla(\lambda_{1}-\Delta)^{-1}f_{1}(u_{1}(s))\Vert_{\infty} ds\\
    & \ + c(N,\xi_{1})\int_{0}^{t} (t-s)^{-\frac{1}{2}}e^{-(t-s)} \Vert u_{2}(s) \Vert_{1} \Vert \nabla(\lambda_{1}-\Delta)^{-1}(f_{1}(u_{1}(s))-f_{1}(u_{2}(s)))\Vert_{\infty} ds\\
    & \ + c(N,\xi_{2}) \int_{0}^{t} (t-s)^{-\frac{1}{2}}e^{-(t-s)}\Vert u_{1}(s)-u_{2}(s) \Vert_{1} \Vert \nabla(\lambda_{2}-\Delta)^{-1}f_{2}(u_{1}(s))\Vert_{\infty} ds\\
    & \ + c(N,\xi_{2})\int_{0}^{t} (t-s)^{-\frac{1}{2}}e^{-(t-s)} \Vert u_{2}(s) \Vert_{1} \Vert \nabla(\lambda_{2}-\Delta)^{-1}(f_{2}(u_{1}(s))-f_{2}(u_{2}(s)))\Vert_{\infty} ds\\
    & \ + \int_{0}^{t}e^{-(t-s)}\Vert G(t-s) \ast (u_{1}(s)-u_{2}(s))\Vert_{1} ds \\
    & \leqslant C(R,l,m,N,\xi_{1},\xi_{2},N,\lambda_{1},\lambda_{2},L_{1},L_{2}) \Vert u_{1}(s)-u_{2}(s) \Vert_{S_{R,T}} t^{\frac{1}{2}} \\
    & \ + (1-e^{-t})\Vert u_{1}(s)-u_{2}(s) \Vert_{S_{R,T}} ,
    \end{align*}
    where  $L_{1}$ and $L_{2}$ are Lipschitz constants of $f_{1}$ and $f_{2}$, respectively. By the similar argument, we have
   \begin{align*}
   \Vert Ku_{1}(t)-Ku_{2}(t)\Vert_{\infty}& \leqslant  C(R,l,m,N,\xi_{1},\xi_{2},N,\lambda_{1},\lambda_{2},L_{1},L_{2}) \Vert u_{1}(s)-u_{2}(s) \Vert_{S_{R,T}} t^{\frac{1}{2}}\\
    &+(1-e^{-t})\Vert u_{1}(s)-u_{2}(s) \Vert_{S_{R,T}}.
    \end{align*}
        For sufficiently small $t$, $K$ is a contraction mapping on $S_{R,T}$.

    \textbf{Claim 4}: $u(t,\cdot) \in X^{\beta}$.

    By the property of the analytic semigroup,\; $T(t)u_{0} \in X^{\beta}$, where $0<\beta < \frac{1}{2}.$ Also by using the fact
    \begin{equation*}
        T(t)\nabla \cdot h = T(\frac{t}{2})(T(\frac{t}{2})\nabla \cdot h) \in X^{\beta}. \quad \forall h\in L^{p}(\mathbb{R}^{N})^{N}, \ 1\leqslant p \leqslant \infty.
    \end{equation*}
    and closedness of $(\Delta-I)^{\beta}$, we obtain
    \begin{align*}
        &\int_{0}^{t} \Vert (\Delta-I)^{\beta}T(\frac{t-s}{2}) T(\frac{t-s}{2})\nabla \cdot (\xi_{1}u(s)\nabla (\lambda_{1}I-\Delta)^{-1} f_{1}(u(s)) \\ &-\xi_{2}u(s) \nabla (\lambda_{2}I-\Delta)^{-1} f_{2}(u(s))\Vert_{p} ds\\
        &\leqslant C_{\beta}\int_{0}^{t}(\frac{t-s}{2})^{-\beta}e^{-\frac{t-s}{2}}\Vert T(\frac{t-s}{2})\nabla \cdot (\xi_{1}u(s)\nabla (\lambda_{1}I-\Delta)^{-1} f_{1}(u(s))\\ & \  -\xi_{2}u(s)\nabla (\lambda_{2}I-\Delta)^{-1}f_{2}(u(s))\Vert_{p}ds \\
        &\leqslant C_{\beta}\int_{0}^{t} (t-s)^{-\beta-\frac{1}{2}}e^{-(t-s)}\Vert \xi_{1}u(s) \nabla (\lambda_{1}I-\Delta)^{-1} f_{1}(u(s))\\ & \ -\xi_{2}u(s)\nabla (\lambda_{2}I-\Delta)^{-1} f_{2}(u(s))\Vert_{p} ds\\
        &\leqslant C_{\beta,\lambda_{1},\lambda_{2},l,m,N,p,\xi_{1},\xi_{2},R}\Gamma(\frac{1}{2}-\beta).
    \end{align*}
    Also, we have
    \begin{align*}
        \int_{0}^{t} \Vert (\Delta-I)^{\beta} T(t-s)u(s)\Vert_{p} ds
         \leqslant C_{p,\beta,R} \int_{0}^{t} (t-s)^{-\beta}e^{-(t-s)} ds
         \leqslant C_{p,\beta,R}  \Gamma(1-\beta).
    \end{align*}
    Thus $u(t,\cdot) \in X^{\beta}$.

\end{proof}

We will now prove the H\"older regularity of the mild solution.

\begin{lemma}\label{lem7}
    Under the same assumption of Lemma \ref{lem6},\ $u(t,x) \in C^{\theta}((0,T);C^{v}(\mathbb{R}^{N}))$ where \ $0<\theta<1 $ , $0< v \ll 1$.
\end{lemma}
\begin{proof}
    \begin{align*}
    u(t)&=T(t)u_{0}\\
    &~~+\int_{0}^{t}T(t-s) \nabla \cdot (\xi_{1}u(s)\nabla (\Delta-\lambda_{1})^{-1} f_{1}(u(s))- \xi_{2}u(s)\nabla (\Delta-\lambda_{2})^{-1} f_{2}(u(s))) ds\\
    &~~ +\int_{0}^{t}T(t-s)u(s)ds\\
    &:=I_{0}(t)+I_{1}(t)+I_{2}(t).
    \end{align*}
    By Lemma \ref{lem6}, we know that $I_{0}(t),I_{1}(t),I_{2}(t) \in X^{\beta}$. Moreover using Lemma \ref{lem5}, it is sufficiently show that $(0,T) \ni t \rightarrow X^{\beta}$ is H\"older continuous.
    \begin{align*}
        \Vert I_{0}(t+h)-I_{0}(t) \Vert_{X^{\beta}}&=\Vert(\Delta-I)^{\beta}(T(h)-I) T(t)u_{0}\Vert_{p}\\
        &=\Vert(T(h)-I)(\Delta-I)^{\beta} T(t)u_{0}\Vert_{p}\\
        &\leqslant C(\beta)h^{1-\beta}\Vert (\Delta-I)T(t)u_{0}\Vert_{p}\\
        &\leqslant C(\beta)h^{1-\beta}t^{-1}e^{-t}\Vert u_{0} \Vert_{p}.
    \end{align*}
    where $0<\beta<1$,$N<p\leqslant \infty$, and $t+h<T$. Let $\delta>0$ be such that $\beta+\delta <\frac{1}{2}$. Then we have
    \begin{align*}
        &\Vert I_{1}(t+h)-I_{1}(t)\Vert_{X^{\beta}} \\
        &=\int_{0}^{t}\Vert (T(h)-I)T(\frac{t-s}{2})T(\frac{t-s}{2}) \nabla \cdot (\xi_{1}u(s)\nabla  (\Delta-\lambda_{1})^{-1} f_{1}(u(s)) \\ & \quad -\xi_{2}u(s)\nabla (\Delta-\lambda_{2})^{-1} f_{2}(u(s)))\Vert_{X^{\beta}} ds\\
        & \ +\int_{t}^{t+h}\Vert T(t+h-s)\nabla \cdot (\xi_{1}u(s)\nabla (\Delta-\lambda_{1})^{-1} f_{1}(u(s))-\xi_{2}u(s)\nabla (\Delta-\lambda_{2})^{-1} f_{2}(u(s)))\Vert_{X^{\beta}} ds\\
        &\leqslant \int_{0}^{t}\Vert (T(h)-I)(\Delta-I)^{\beta}T(\frac{t-s}{2})T(\frac{t-s}{2})\nabla \cdot (\xi_{1}u(s)\nabla (\Delta-\lambda_{1})^{-1} f_{1}(u(s))\\ & \quad -\xi_{2}u(s)\nabla (\Delta-\lambda_{2})^{-1}  f_{2}(u(s)))\Vert_{p} ds\\
        & \ +\int_{t}^{t+h} \Vert (\Delta-I)^{\beta}T(\frac{t+h-s}{2})T(\frac{t+h-s}{2})\nabla \cdot (\xi_{1}u(s)\nabla (\Delta-\lambda_{1})^{-1} f_{1}(u(s))\\ & \quad -\xi_{2}u(s)\nabla (\Delta-\lambda_{2})^{-1}  f_{2}(u(s))\Vert_{p} ds\\
        &\leqslant C(R,l,m,\lambda_{1},\lambda{2},\xi_{1},\xi_{2},p,N,\delta,\beta)h^{\delta}\int_{0}^{t} \left(\frac{t-s}{2}\right)^{-\delta-\beta-\frac{1}{2}}e^{-(t-s)}ds\\
        & \ +C(R,l,m,\lambda_{1},\lambda{2},\xi_{1},\xi_{2},p,N,\beta)\int_{t}^{t+h}\left(\frac{t+h-s}{2}\right)^{-\beta-\frac{1}{2}}e^{-(t+h-s)}ds\\
        &\leqslant C(R,l,m,\lambda_{1},\lambda{2},\xi_{1},\xi_{2},p,N,\alpha)h^{\delta}\Gamma(\frac{1}{2}-\delta-\beta)+ C(R,l,m,\lambda_{1},\lambda{2},\xi_{1},\xi_{2},p,N,\alpha)h^{\frac{1}{2}-\beta}
    \end{align*}
     and
    \begin{align*}
        &\Vert I_{2}(t+h)-I_{2}(t)\Vert_{X^{\beta}}\\
        &\leqslant \int_{0}^{t} \Vert (\Delta-I)^{\beta}(T(h)-I)T(t-s)u(s)\Vert_{p}ds+\int_{t}^{t+h}\Vert (\Delta-I)^{\beta}T(t+h-s)u(s)\Vert_{p}ds\\
        &\leqslant C(\delta,\beta,p,R)h^{\delta}\int_{0}^{t}(t-s)^{-\beta-\delta}e^{-(t-s)}ds+C(\beta,p,R)\int_{t}^{t+h}(t+h-s)^{-\beta}e^{-(t+h-s)}ds\\
        &\leqslant C(\delta,\beta,p,R)h^{\delta}\Gamma(1-\beta-\delta)+C(\beta,p,R)h^{1-\beta}.
    \end{align*}
    This complete the proof of Lemma \ref{lem7}.
\end{proof}

We briefly introduce the parabolic Cauchy problem such that
\begin{align*}
    &Lu=f(x,t) \quad \text{in} \quad \mathbb{R}^{N} \times (0,T],\\
    &u(x,0)=\phi(x) \quad \text{on} \quad \mathbb{R}^{N} ,
\end{align*}
where $L$ is the uniformly parabolic operator defined by (\ref{1}). If $f(x,t),\phi(x)$ are exponentially bounded (for more Details conditions,see \cite[Ch1]{fri}) and
coefficients of $L$ satisfying the Lemma \ref{lem1}, then we construct solution by the following form
\begin{equation*}
    u(x,t)=\int_{\mathbb{R}^{N}} Z(x,\xi,t,0)\phi(\xi)d\xi -\int_{0}^{t}\int_{\mathbb{R}^{N}}Z(x,\xi,t,\tau)f(\xi,\tau)d\xi d\tau.
\end{equation*}
where $Z$ is the parabolic fundamental solution (see \cite{fri,rus} for more information of the parabolic fundamental solution).

\begin{lemma}\label{lem8}
Assume that the nonnegative initial data,$u_{0}$, of the problem (1.1) belongs to $ L^{1}(\mathbb{R}^{N})\cap BUC(\mathbb{R}^{N})$ with $N\geqslant2$.  then (1.1) has a unique nonnegative local classic solution $u \in C^{1,2}((0,T_{max})\times \mathbb{R}^{N}) \cap C^{1}((0,T_{max};L^{p}(\mathbb{R}^{N}))\cap C([0,T_{max});L^{1}(\mathbb{R}^{N}) \cap BUC(\mathbb{R}^{N}) )\cap  C^{\theta}((0,T_{max});C_{unif}^{\nu}(\mathbb{R}^{N}))$ and $\lim_{t \rightarrow 0^{+}} \Vert u(t,x)-u_{0}(x) \Vert_{X}=0$
for $0 < \theta,\nu <1$.
 If $T_{max} <\infty,$ then $\lim_{t\rightarrow T_{max}^{-}}\Vert u(t)\Vert_{\infty}=\infty$.\;In addition $\Vert u(x,t)\Vert_{1}= \Vert u_{0}(x) \Vert_{1}$.
\end{lemma}
\begin{proof}
    Let $0<t_{1}<T_{max}$ be fixed and consider the following parabolic Cauchy problem:
    \begin{equation}\label{3}
    \begin{aligned}
      &\tilde{u_{t}}=(\Delta-I)\tilde{u}+F(t,\tilde{u})\quad \text{in} \quad \mathbb{R}^{N}\times (0,T], \\
      &\tilde{u}(x,0)=u_{1}(x):=u(t_{1}) \quad \text{on} \quad \mathbb{R}^{N}.
    \end{aligned}
    \end{equation}
    where $F(t,\tilde{u})=-\xi_{1}\nabla v(t+t_{1})\nabla \tilde{u}(t)+\xi_{2}\nabla w(t+t_{1})\nabla \tilde{u}(t)+\tilde{u}(t)+\xi_{2}\Delta w(t+t_{1})\tilde{u}(t)-\xi_{1}\Delta v(t+t_{1})\tilde{u}(t).$ By Lemma \ref{lem6} and \ref{lem7} and assumption of the $f_{1},f_{2}$, we obtain that
    \begin{align*}
    & u(t+t_{1},\cdot),f_{1}(u(t+t_{1},\cdot)),f_{2}(u(t+t_{1},\cdot)),v(t+t_{1},\cdot),w(t+t_{1},\cdot),\frac{\partial w(t+t_{1},\cdot)}{\partial x_{i}},\\ &\frac{\partial v(t+t_{1},\cdot)}{\partial x_{i}},\frac{\partial^{2}v(t+t_{1},\cdot)}{\partial x_{i}\partial x_{j}},\frac{\partial^{2}w(t+t_{1},\cdot)}{\partial x_{i}\partial x_{j}} \in C^{\frac{\alpha}{2},\alpha}([0,T],\mathbb{R}^{N}),
    \end{align*}
     for $t\geqslant 0$, $0<\alpha<1$, $i,j=1,2,\cdot \cdot \cdot ,N$ and $T:=T_{max}-t_{1}-\epsilon$, where $0<\epsilon<T_{max}-t_{1}$. Thus by Lemma \ref{lem1} and \cite[Theorem 16 Ch1]{fri}, we obtain the unique classical solution for \eqref{3},
     \begin{equation*}
        \tilde{u}(x,t)=\int_{\mathbb{R}^{N}} Z(x,\xi,t,0)u_{1}(\xi)d\xi,
     \end{equation*}
     where $Z$ is the parabolic fundamental solution. By \cite[Theorem 11 Ch1]{fri}, $\lim_{t \rightarrow 0}\Vert \tilde{u}(t)-u_{1}\Vert_{\infty}=0$. By using the Lemma \ref{lem1}, the mappings
     \begin{align}
        t \mapsto \tilde{u}(t) \in L^{p} \cap BUC(\mathbb{R}^{N}) \quad \text{for} \quad T>t>0, 1\leqslant p \leqslant \infty, \label{4} \\
         t \mapsto \frac{\partial \tilde{u}(t)}{\partial x_{i}} \in L^{p}\cap BUC(\mathbb{R}^{N}) \quad \text{for} \quad T>t>0, 1\leqslant p \leqslant \infty, \label{5}
     \end{align}
    are locally H\"older continuous in time. By using Lemma \ref{lem6}, \ref{lem7} and \eqref{4},\eqref{5}, $F(t,\tilde{u})$ is locally H\"older continuous for $T>t>0$. By using the Lemma 1.1, we obtain
     \begin{equation*}
        F(t,\tilde{u}) \in L^{1}(0,T;L^{p}(\mathbb{R}^{N})),
     \end{equation*}
     for $1\leqslant p \leqslant\infty.$ Thus by the \cite[Corollary 3.3 p113]{pazy},
     $ \tilde{u}(t) $ is also a mild solution of \eqref{3} and then satisfies the following integral equation,
     \begin{align*}
        \tilde{u}(t)=T(t)u_{1}+&\int_{0}^{t}T(t-s)(-\xi_{1} \nabla v(s+t_{1})\nabla \tilde{u}(s) +\xi_{2}\nabla w(s+t_{1})\nabla \tilde{u}(s)+\\&\tilde{u}(s)+\xi_{2} \Delta w(s+t_{1})\tilde{u}(s)-\xi_{1}\Delta v(s+t_{1})\tilde{u}(s))ds
     \end{align*}
     for $0\leqslant t < T$.
     By using the fact
     \begin{align*}
        \nabla \tilde{u}(s)\nabla v(s+t_{1})=&\nabla \cdot (\tilde{u}(s)\nabla v(s+t_{1}))-\tilde{u}(s)\Delta v(s+t_{1})\\
        =&\nabla \cdot (\tilde{u}(s)\nabla v(s+t_{1}))-\tilde{u}(s)(\lambda_{1}v(s+t_{1})-f_{1}(u(s+t_{1})))
     \end{align*}
     and
     \begin{align*}
        \nabla \tilde{u}(s)\nabla w(s+t_{1})=&\nabla \cdot (\tilde{u}(s)\nabla w(s+t_{1}))-\tilde{u}(s)\Delta w(s+t_{1})\\
        =&\nabla \cdot (\tilde{u}(s)\nabla w(s+t_{1}))-\tilde{u}(s)(\lambda_{2}w(s+t_{1})-f_{2}(u(s+t_{1}))),
     \end{align*}
     we obtain
     \begin{align*}
        \tilde{u}(t)=&T(t)u_{1}+\int_{0}^{t} T(t-s)\nabla \cdot (-\xi_{1} \nabla \cdot (\tilde{u}(s)\nabla v(s+t_{1}))\\& \quad +\xi_{2}\nabla \cdot (\tilde{u}(s)\nabla w(s+t_{1}))) ds \nonumber
         +\int_{0}^{t} T(t-s)\tilde{u}(s) ds.
     \end{align*}
     On the other hand, we have that
     \begin{align*}
        u(t+t_{1})=&T(t)u_{1}+\int_{0}^{t}T(t-s)(-\xi_{1} \nabla \cdot  (u(s+t_{1})\nabla v(s+t_{1})\\ & \ + \xi_{2}\nabla \cdot (u(s+t_{1})\nabla w(s+t_{1})))ds
        +\int_{0}^{t}T(t-s)u(s+t_{1})ds
     \end{align*}
     for $0<t<T_{max}-t_{1}$ .
     Thus for $0<t<T_{\epsilon}<T=T_{max}-t_{1}-\epsilon$, we obtain
     \begin{align*}
        \Vert & \tilde{u}(t)-u(t+t_{1}) \Vert_{\infty}\\
        &\leqslant \xi_{1}\int_{0}^{t} \Vert T(t-s) \nabla \cdot ((u(s+t_{1})-\tilde{u}(s))\nabla v(s+t_{1}))\Vert_{\infty} ds \\
        & \quad+\xi_{2}\int_{0}^{t} \Vert T(t-s) \nabla \cdot ((u(s+t_{1})-\tilde{u}(s))\nabla w(s+t_{1}))\Vert_{\infty} ds\\
        & \quad+\int_{0}^{t}\Vert T(t-s)(u(s+t_{1})-\tilde{u}(s))\Vert_{\infty} ds \\
        & \leqslant \xi_{1} \sup_{0\leqslant t \leqslant T_{\epsilon}}\Vert \nabla v(t+t_{1}) \Vert _{\infty} \int_{0}^{t} (t-s)^{-\frac{1}{2}}e^{-(t-s)}\Vert u(s+t_{1})-\tilde{u}(s)\Vert_{\infty} ds \\
        & \quad+ \xi_{2} \sup_{0\leqslant t \leqslant T_{\epsilon}}\Vert \nabla w(t+t_{1}) \Vert _{\infty} \int_{0}^{t} (t-s)^{-\frac{1}{2}}e^{-(t-s)}\Vert u(s+t_{1})-\tilde{u}(s)\Vert_{\infty} ds\\
        &\quad +\sqrt{T}\int_{0}^{t} (t-s)^{-\frac{1}{2}}e^{-(t-s)}\Vert u(s+t_{1})-\tilde{u}(s)\Vert_{\infty} ds
     \end{align*}
     By  Gronwall's inequality, we conclude that
     \begin{equation*}
        u(t+t_{1})=\tilde{u}(t) \quad 0\leqslant t \leqslant T_{\epsilon}.
     \end{equation*}
     Since $\epsilon,t_{1}$ are arbitrary positive constants, $u$ is a classical solution on $(0,T_{max})$.
     Since $u_{0} \geqslant 0$, by the comparison principle for Parabolic equation (see \cite[Lemma 5 p43]{fri} or \cite[Corollary 4.2.1 p136]{nonlinear}), we obtain $u(x,t)\geqslant 0$.
     Since $ v=(\lambda_{1}I-\Delta)^{-1}f_{1}(u), w=(\lambda_{2}I-\Delta)^{-1} f_{2}(u)$, $v$ and $w$ are also nonnegative.

     Next, we will show the uniqueness of the solution. Suppose that for given nonnegative initial data $u_{0} \in L^{1}(\mathbb{R}^{N})\cap BUC(\mathbb{R}^{N})$, \ $(u_{1},v_{1},w_{1})$ and $(u_{2},v_{2},w_{2})$ are two classical solutions of (1.1)
     on $\mathbb{R}^{N} \times [0,T)$ satisfying of the Lemma \ref{lem8}. Let $0<t_{1}<T'<T$ be fixed. Then $\sup_{0\leqslant t \leqslant T'}(\Vert u_{1}(t,\cdot) \Vert_{\infty}+\Vert u_{2}(t,\cdot)\Vert_{\infty})<\infty$. Consider the following Cauchy problem
     \begin{align*}
        &\tilde{u_{t}}=(\Delta-I)\tilde{u}+F(t,\tilde{u})\quad \text{in} \quad \mathbb{R}^{N}\times (0,T'],\nonumber \\
       &\tilde{u}(x,0)=u_{i}(t_{1},\cdot) \quad \text{on} \quad \mathbb{R}^{N},
     \end{align*}
     where $F(t,\tilde{u})=-\xi_{1}\nabla v_{i}(t+t_{1},x)\nabla \tilde{u}(t,x)+\xi_{2}\nabla w_{i}(t+t_{1},x)\nabla \tilde{u}(t,x)+\tilde{u}(t,x)+\xi_{2}\Delta w_{i}(t+t_{1},x)\tilde{u}(t,x)-\xi_{1}\Delta v_{i}(t+t_{1},x)\tilde{u}(t,x) $ for $i=1,2$. By  the uniqueness of the Cauchy problem, classical solutions $(u_{i},v_{i},w_{i})$ satisfying of the Lemma \ref{lem8} becomes the following mild solutions for $ t_{1}\leqslant t \leqslant T'$,
     \begin{align*}
        u_{i}(t)&=T(t-t_{1})u_{i}(t_{1})+\int_{t_{1}}^{t}T(t-s)(-\xi_{1}\nabla \cdot (u_{i}(s)\nabla v_{i}(s))+\xi_{2}\nabla \cdot (u_{i}(s)\nabla w_{i}(s))ds \\&+\int_{t_{1}}^{t}T(t-s)u_{i}(s)ds,
     \end{align*}
     for $i=1,2$. Then it follows that
     \begin{align*}
        \Vert & u_{1}(t)-u_{2}(t) \Vert_{\infty}\\ &\leqslant \Vert u_{1}(t_{1})-u_{2}(t_{1})\Vert_{\infty}+c_{1}\xi_{1}\int_{t_{1}}^{t}(t-s)^{-\frac{1}{2}}e^{-(t-s)}\Vert u_{1}(s)\nabla v_{1}(s)-u_{2}(s)\nabla v_{2}(s) \Vert_{\infty} ds \\
        &\quad +c_{2}\xi_{2}\int_{t_{1}}^{t}(t-s)^{-\frac{1}{2}}e^{-(t-s)}\Vert u_{1}(s)\nabla w_{1}(s)-u_{2}(s)\nabla w_{2}(s) \Vert_{\infty} ds\\ & \quad+\int_{t_{1}}^{t}e^{-(t-s)}\Vert u_{1}(s)-u_{2}(s)\Vert_{\infty} ds\\
        &\leqslant \Vert u_{1}(t_{1})-u_{2}(t_{1})\Vert_{\infty}  \\
        &\quad+c_{1}\xi_{1}\int_{t_{1}}^{t}(t-s)^{-\frac{1}{2}}e^{-(t-s)}\Vert u_{1}(s)-u_{2}(s) \Vert_{\infty}\Vert \nabla v_{1}(s)\Vert_{\infty} ds \\
        &\quad+c_{1}\xi_{1}\int_{t_{1}}^{t}(t-s)^{-\frac{1}{2}}e^{-(t-s)} \Vert \nabla v_{1}(s)-\nabla v_{2}(s) \Vert_{\infty} \Vert u_{2}(s)\Vert_{\infty} ds \\
        &\quad+c_{2}\xi_{2}\int_{t_{1}}^{t}(t-s)^{-\frac{1}{2}}e^{-(t-s)}\Vert u_{1}(s)-u_{2}(s) \Vert_{\infty}\Vert \nabla w_{1}(s)\Vert_{\infty} ds\\
        &\quad+c_{2}\xi_{2}\int_{t_{1}}^{t}(t-s)^{-\frac{1}{2}}e^{-(t-s)} \Vert \nabla w_{1}(s)-\nabla w_{2}(s) \Vert_{\infty} \Vert u_{2}(s)\Vert_{\infty} ds\\
        &\quad+\int_{t_{1}}^{t}e^{-(t-s)}\Vert u_{1}(s)-u_{2}(s)\Vert_{\infty} ds\\
        & \leqslant \Vert u_{1}(t_{1})-u_{2}(t_{1})\Vert_{\infty} \\
        &\quad+c_{1}\xi_{1} \sup_{0\leqslant \tau \leqslant T'} \Vert \nabla v_{1}(\tau) \Vert_{\infty} \int_{t_{1}}^{t}(t-s)^{-\frac{1}{2}}e^{-(t-s)}\Vert u_{1}(s)-u_{2}(s) \Vert_{\infty} ds\\
        &\quad+c_{1}\xi_{1} \sup_{0\leqslant \tau \leqslant T'} \Vert u_{2}(s)\Vert_{\infty} \int_{t_{1}}^{t}(t-s)^{-\frac{1}{2}}e^{-(t-s)} \Vert f_{1}(u_{1}(s))-f_{1}(u_{2}(s))\Vert_{\infty} ds\\
        &\quad+c_{2}\xi_{2} \sup_{0\leqslant \tau \leqslant T'} \Vert \nabla w_{1}(\tau) \Vert_{\infty} \int_{t_{1}}^{t}(t-s)^{-\frac{1}{2}}e^{-(t-s)}\Vert u_{1}(s)-u_{2}(s) \Vert_{\infty} ds\\
        &\quad+c_{2}\xi_{2} \sup_{0\leqslant \tau \leqslant T'} \Vert u_{2}(s)\Vert_{\infty} \int_{t_{1}}^{t}(t-s)^{-\frac{1}{2}}e^{-(t-s)} \Vert f_{2}(u_{1}(s))-f_{2}(u_{2}(s))\Vert_{\infty} ds\\
        &\quad+\sqrt{T'} \int_{t_{1}}^{t}(t-s)^{-\frac{1}{2}}e^{-(t-s)}\Vert u_{1}(s)-u_{2}(s) \Vert_{\infty} ds\\
        & \leqslant \Vert u_{1}(t_{1})-u_{2}(t_{1})\Vert_{\infty} \\
        &\quad+M\int_{t_{1}}^{t}(t-s)^{-\frac{1}{2}}e^{-(t-s)}\Vert u_{1}(s)-u_{2}(s) \Vert_{\infty} ds,
     \end{align*}
     where $M=c_{1}\xi_{1}\sup_{0\leqslant \tau \leqslant T'} \Vert \nabla v_{1}(\tau) \Vert_{\infty} +c_{2}\xi_{2}\sup_{0\leqslant \tau \leqslant T'} \Vert \nabla w_{1}(\tau) \Vert_{\infty}+
     (L_{1}c_{1}\xi_{1}+L_{2}c_{2}\xi_{2}) \sup_{0\leqslant \tau \leqslant T'} \Vert u_{2}(s)\Vert_{\infty} <\infty $. Let $t_{1} \rightarrow 0$, we have
     \begin{equation*}
        \Vert u_{1}(t)-u_{2}(t)\Vert_{\infty} \leqslant M \int_{0}^{t}(t-s)^{-\frac{1}{2}}e^{-(t-s)}\Vert u_{1}(s)-u_{2}(s) \Vert_{\infty} ds.
     \end{equation*}
     Using Gronwall's inequality, we obtain $u_{1}(t)=u_{2}(t)$ for $0 \leqslant t \leqslant T'$.
     Since $T'<T$ was arbitrary, $u_{1}(t)=u_{2}(t)$ for $0 \leqslant t \leqslant T $.

     Next, by integration by (1.1) over the $\mathbb{R}^{N}$, we obtain the conservation law, $\Vert u(x,t) \Vert _{1} =\Vert u_{0}(x) \Vert_{1}$.

\end{proof}

\begin{lemma}\label{lem9}
        For $N \geqslant 2 $, assume that $0\leqslant f_{1}(s) \leqslant c_{1}s^{l},f_{2}(s)=s^{m}$, where $l > \frac{2}{N},m \geqslant 1$ and $l <m$. Then the solution of (1.1) has following property:
        \begin{align}
            &\Vert u \Vert_{p} \leqslant C_{1} \quad \text{for} \ 1\leqslant p < \infty, \label{12} \\
            &\Vert v \Vert_{p} \leqslant C_{2} \quad \text{for} \ \frac{1}{l} \leqslant p \leqslant \infty ,\ \text{when} \ l<1,\label{13} \\
            &\Vert v \Vert_{p} \leqslant C_{3} \quad \text{for} \ 1\leqslant p \leqslant \infty ,\ \text{when} \ l\geqslant 1, \label{14}\\
            &\Vert w \Vert_{p} \leqslant C_{4} \quad \text{for} \ 1\leqslant p \leqslant \infty,\label{15}\\
            & \Vert \nabla v \Vert_{\infty} \leqslant C_{5}, \quad \Vert \nabla w \Vert_{\infty} \leqslant C_{6}, \label{16}
        \end{align}
        where $C_{1},C_{2},C_{3},C_{4},C_{5},C_{6}$ are for some positive constants depending on \\ $\Vert u_{0} \Vert_{1}, N, m,l,\lambda_{1},\lambda_{2},c_{1},\xi_{1},\xi_{2},p$.
\end{lemma}
\begin{proof}
For $r> max\{2,\frac{m}{N-1}\}$, we multiply the first equation of (1.1) with $u^{r-1}$ and integrate over $\mathbb{R}^{N}$. Then
    \begin{align*}
        \frac{1}{r} \frac{d}{dt} \Vert u \Vert _{r}^{r}&=-(r-1)\int_{\mathbb{R}^{N}} u^{r-2}\vert \nabla u\vert^{2} dx+\xi_{1}(r-1)\int_{\mathbb{R}^{N}} u^{r-1}\nabla u \nabla v dx \\&-\xi_{2} (r-1)\int_{\mathbb{R}^{N}} u^{r-1}\nabla u \nabla w dx.
    \end{align*}
    We also multiply the second and third equation of (1.1) with $u^{r}$ and integrate by part over $\mathbb{R}^{N}$. Then we have
    \begin{equation*}
        \int_{\mathbb{R}^{N}} u^{r-1} \nabla v \nabla u dx =-\frac{\lambda_{1}}{r} \int_{\mathbb{R}^{N}} vu^{r}dx +\frac{1}{r}\int_{\mathbb{R}^{N}} f_{1}(u)u^{r} dx
    \end{equation*}
    and
    \begin{equation*}
        \int_{\mathbb{R}^{N}} u^{r-1} \nabla w \nabla u dx =-\frac{\lambda_{2}}{r} \int_{\mathbb{R}^{N}} wu^{r}dx +\frac{1}{r}\int_{\mathbb{R}^{N}} f_{2}(u)u^{r} dx.
    \end{equation*}
    Thus
    \begin{equation}\label{6}
    \begin{aligned}
        \frac{1}{r}\frac{d}{dt} \Vert u \Vert _{r}^{r}&=-\frac{4(r-1)}{r^{2}} \int_{\mathbb{R}^{N}} \vert \nabla u^{\frac{r}{2}} \vert ^{2} dx
         \\
        &-\frac{\xi_{1}(r-1)\lambda_{1}}{r} \int_{\mathbb{R}^{N}} vu^{r}dx +\frac{\xi_{1}(r-1)}{r} \int_{\mathbb{R}^{N}}f_{1}(u)u^{r} dx
        \\
        &+\frac{\xi_{2}(r-1)\lambda_{2}}{r}\int_{\mathbb{R}^{N}} wu^{r}dx-\frac{\xi_{2}(r-1)}{r} \int_{\mathbb{R}^{N}}f_{2}(u)u^{r} dx.
    \end{aligned}
    \end{equation}
    Using \eqref{6} and nonnegativity of the solutions, we obtain
    \begin{equation}\label{8}
    \begin{aligned}
        \frac{1}{r}\frac{d}{dt} \Vert u \Vert _{r}^{r}&\leqslant-\frac{4(r-1)}{r^{2}} \int_{\mathbb{R}^{N}} \vert \nabla u^{\frac{r}{2}} \vert ^{2} dx
         +\frac{\xi_{1}c_{1}(r-1)}{r} \int_{\mathbb{R}^{N}}u^{l+r} dx  \\
        &+\frac{\xi_{2}(r-1)\lambda_{2}}{r}\int_{\mathbb{R}^{N}} wu^{r}dx-\frac{\xi_{2}(r-1)}{r} \int_{\mathbb{R}^{N}}u^{m+r} dx.
    \end{aligned}
    \end{equation}
    By using the interpolation inequality, conservation law and Young's inequality, we have
    \begin{equation}\label{9}
        \Vert u \Vert_{l+r}^{l+r} \leqslant \Vert u_{0} \Vert_{1}^{(l+r)(1-\theta_{1})}\Vert u \Vert_{m+r}^{\theta_{1}(l+r)} \leqslant C(\Vert u_{0} \Vert_{1},l,r,m) +\epsilon_{1} \Vert u \Vert_{m+r}^{m+r}.
    \end{equation}
    where $0<\theta_{1} < 1$.

    We also have following inequality by  Young's inequality,
    \begin{equation}\label{10}
        \int_{\mathbb{R}^{N}} wu^{r} dx \leqslant \epsilon_{2}\Vert u \Vert_{m+r}^{m+r}+c(\epsilon_{2})\Vert w \Vert_{\frac{m+r}{m}}^{\frac{m+r}{m}}.
    \end{equation}
    Note that $ w= \Gamma_{\lambda_{2}}\ast u^{m}$ where $ \Gamma_{\lambda_{2}}$ is the Bessel potential.
    Thus by the $L_{p}-L_{q}$ convolution inequality, we get that
    \begin{equation*}
        c(\epsilon_{2})\Vert w \Vert_{\frac{m+r}{m}}^{\frac{m+r}{m}} \leqslant c(\epsilon_{2})\Vert \Gamma_{\lambda_{2}}\Vert_{\frac{N}{N-1}}^{(1-\theta_{2})\frac{m+r}{m}}\Vert u^{m}\Vert_{q}^{\theta_{2}\frac{m+r}{m}},
    \end{equation*}
    where $q=\frac{N(m+r)}{mN+m+r},0 <\theta_{2}<1$. Moreover
    \begin{equation*}
        \Vert u^{m} \Vert_{q}^{\frac{m+r}{m}}=\Vert u \Vert_{qm}^{m+r} \leqslant \Vert u_{0} \Vert_{1}^{(1-\theta_{3})(m+r)}\Vert u \Vert_{m+r}^{\theta_{3}(m+r)},
    \end{equation*}
    where $0 <\theta_{3}<1$. Using above inequalities and Young's inequality, we have
    \begin{equation}\label{11}
    \begin{aligned}
        c(\epsilon_{2})\Vert w \Vert_{\frac{m+r}{m}}^{\frac{m+r}{m}}& \leqslant c(\epsilon_{2})\Vert \Gamma_{\lambda_{2}}\Vert_{\frac{N}{N-1}}^{(1-\theta_{2})\frac{m+r}{m}}\Vert u_{0} \Vert_{1}^{\theta_{2}(1-\theta_{3})(m+r)}\Vert u \Vert_{m+r}^{\theta_{2}\theta_{3}(m+r)} \\
        &\leqslant \epsilon_{3} \Vert u \Vert_{m+r}^{m+r}+c(\epsilon_{2},\epsilon_{3},\Vert u_{0}\Vert_{1},m,N,\lambda_{2},r).
    \end{aligned}
    \end{equation}
    Thus from \eqref{10} and \eqref{11} we have
    \begin{equation}\label{7}
         \int_{\mathbb{R}^{N}} wu^{r} dx \leqslant \epsilon_{2}\Vert u \Vert_{m+r}^{m+r}+ \epsilon_{3} \Vert u \Vert_{m+r}^{m+r}+c(\epsilon_{2},\epsilon_{3},\Vert u_{0}\Vert_{1},m,N,\lambda_{2},r).
    \end{equation}
    Inserting \eqref{9} and \eqref{7} into \eqref{8}, and then choosing $\epsilon_{1},\epsilon_{2},\epsilon_{3}$  small enough,  we obtain
    \begin{equation}\label{318}
    \begin{aligned}
        &\frac{1}{r}\frac{d}{dt}\Vert u \Vert_{r}^{r} \\ &\leqslant -\frac{4(r-1)}{r^{2}} \int_{\mathbb{R}^{N}} \vert \nabla u^{\frac{r}{2}} \vert ^{2} dx \\& \quad + \left(\frac{\xi_{2}\lambda_{2}(r-1)}{r}(\epsilon_{2}+\epsilon_{3})+\epsilon_{1}c_{1}\frac{\xi_{1}(r-1)}{r} -\frac{\xi_{2}(r-1)}{r}\right)\Vert u \Vert_{m+r}^{m+r} \\ & \quad +C(\Vert u_{0}\Vert_{1},l,m,N,\xi_{1},\xi_{2},c_{1},r)
       \\
        & \leqslant -\frac{4(r-1)}{r^{2}} \int_{\mathbb{R}^{N}} \vert \nabla u^{\frac{r}{2}} \vert^{2} dx -C(\lambda_{2},\xi_{1},\xi_{2},N,l,m,c_{1},r)\Vert u \Vert_{m+r}^{m+r} \\& \quad + C(\Vert u_{0}\Vert_{1},l,m,N,\xi_{1},\xi_{2},c_{1},r).
    \end{aligned}
    \end{equation}
    Again using the interpolation inequality and Young's inequality, we have
    \begin{equation}\label{319}
        \Vert u \Vert_{r}^{r} \leqslant \Vert u_{0} \Vert_{1}^{r\theta_{4}}\Vert u \Vert_{m+r}^{(1-\theta_{4})r} \leqslant \epsilon_{4} \Vert u \Vert_{m+r}^{m+r} +C(r,\Vert u_{0} \Vert_{1}),
    \end{equation}
    where $0<\theta_{4}<1$. Adding \eqref{318} to \eqref{319}, we obtain
    \begin{equation*}
        \frac{1}{r}\frac{d}{dt}\Vert u \Vert_{r}^{r} +\Vert u \Vert_{r}^{r} \leqslant C(\Vert u_{0}\Vert_{1},l,m,N,\xi_{1},\xi_{2},c_{1},r),
    \end{equation*}
    for sufficiently small $\epsilon_{4}>0$. By the comparison principle of ordinary differential equations, we have
    \begin{equation*}
        \Vert u \Vert_{p} \leqslant  C(\Vert u_{0}\Vert_{1},l,m,N,\xi_{1},\xi_{2},c_{1},p),
    \end{equation*}
    where $1\leqslant p <\infty$. Thus we have \eqref{12}.

    If $l<1$ , $q \geqslant \frac{1}{l}$, then
    \begin{equation}\label{26}
    \begin{aligned}
        \Vert v \Vert_{q} = \Vert (\lambda_{1}I-\Delta)^{-1}&f_{1}(u)\Vert_{q} = \Vert \int_{0}^{\infty} e^{-\lambda_{1}t} G(\cdot,t)\ast f_{1}(u) dt\Vert_{q} \\
        &\leqslant C(N,p,q)\Vert f_{1}(u) \Vert_{p}\int_{0}^{\infty} e^{-\lambda_{1}t} t^{-\frac{N}{2}(\frac{1}{p}-\frac{1}{q})} dt.
    \end{aligned}
    \end{equation}
    Taking $p \geqslant \frac{1}{l}$ and $ 0 \leqslant \frac{1}{p}-\frac{1}{q} < \frac{2}{N}$, we have
    \begin{equation*}
        \Vert v \Vert_{q} \leqslant C(c_{1},l,N,\lambda_{1}), \quad \text{for} \ \frac{1}{l} \leqslant q \leqslant \infty.
    \end{equation*}
    For $ l \geqslant 1 $, we have
    \begin{multline}\label{17}
        \Vert v \Vert_{q} = \Vert (\lambda_{1}I-\Delta)^{-1}f_{1}(u)\Vert_{q}  = \Vert \Gamma_{\lambda_{1}} \ast f_{1}(u) \Vert_{q}
          \\
          \leqslant \Vert \Gamma_{\lambda_{1}} \Vert_{\frac{N}{N-1}} \Vert f_{1}(u) \Vert_{\frac{qN}{q+N}}
          \leqslant C(c_{1},\lambda_{1},N,q,l,\Vert u_{0} \Vert_{1}),
    \end{multline}
    for $1\leqslant q \leqslant \infty$. By the same argument as \eqref{17}, we obtain
    \begin{equation*}
        \Vert w \Vert_{q} \leqslant C(\lambda_{2},N,q,m,\Vert u_{0} \Vert_{1}),
    \end{equation*}
    for $1 \leqslant q \leqslant \infty.$ Thus we have \eqref{13},\eqref{14} and \eqref{15}.

    On the other hand, when $\frac{2}{N} < l < 1$,
    \begin{equation}\label{18}
    \begin{aligned}
        \Vert \nabla v \Vert_{\infty} &= \Vert \nabla (\lambda_{1}-\Delta)^{-1}f_{1}(u)\Vert_{\infty}
        \\
        &\leqslant \int_{0}^{\infty} e^{-\lambda_{1}t}\Vert \nabla G(\cdot,t) \ast f_{1}(u) \Vert_{\infty} dt \\
        &\leqslant C(p,N) \int_{0}^{\infty} e^{-\lambda_{1}t} t^{-\frac{1}{2}-\frac{N}{2}\frac{1}{p}}\Vert f_{1}(u)\Vert_{p} dt\\
        &\leqslant C(c_{1},p,N,l,\lambda_{1})\Gamma(\frac{1}{2}-\frac{N}{2p}),
    \end{aligned}
    \end{equation}
    where $p$ is sufficiently large. Similar to \eqref{18}, when $l\geqslant 1, m \geqslant 1$, we have
    \begin{equation*}
        \Vert \nabla v \Vert_{\infty} \leqslant  C(c_{1},\lambda_{1},l,N,\Vert u_{0} \Vert_{1})\quad \text{and}\quad \Vert \nabla w \Vert_{\infty} \leqslant C(\lambda_{2},m,N,\Vert u_{0} \Vert_{1}).
    \end{equation*}
    Thus we obtain \eqref{16}.

\end{proof}

\begin{lemma}\label{lem10}
        For $N \geqslant 2 $,  if $0\leqslant f_{1}(s) \leqslant c_{1}s^{l},0\leqslant f_{2}(s) \leqslant c_{2}s^{m}\ \text{for} \ l=m <\frac{2}{N}$, then it holds that
    \begin{align}
        &\Vert u \Vert_{p}\leqslant C_{7}, \quad \text{for} \ 1\leqslant p < \infty, \label{19} \\
        &\Vert (v,w) \Vert_{p} \leqslant C_{8}, \quad \text{for} \ \frac{1}{l}=\frac{1}{m} \leqslant p \leqslant \infty,\label{20}\\
        &\Vert \nabla v \Vert_{\infty}  \leqslant C_{9}, \quad \Vert \nabla w \Vert_{\infty} \leqslant C_{10},\label{21}
    \end{align}
    where positive constants \ $C_{7},C_{8},C_{9},C_{10}$ depend on the $l,m,\lambda_{1},\lambda_{2},\Vert u_{0} \Vert_{1},\xi_{1},\xi_{2},c_{1},c_{2},N$.
\end{lemma}
\begin{proof}
    By \eqref{6} and  Young's inequality, we have
    \begin{equation}\label{24}
    \begin{aligned}
        \frac{1}{r}\frac{d}{dt}\Vert u \Vert_{r}^{r}&\leqslant -\frac{4(r-1)}{r^{2}}\int_{\mathbb{R}^{N}} \vert \nabla u^{\frac{r}{2}} \vert^{2} dx \\
        & \quad +\frac{c_{1}\xi_{1}(r-1)}{r}\Vert u \Vert_{l+r}^{l+r}+\frac{\xi_{2}\lambda_{2}(r-1)}{r}\int_{\mathbb{R}^{N}} wu^{r} dx\\
        &\leqslant -\frac{4(r-1)}{r^{2}}\int_{\mathbb{R}^{N}} \vert \nabla u^{\frac{r}{2}} \vert^{2} dx +\frac{\xi_{2}\lambda_{2}(r-1)}{r}\frac{l}{l+r}\Vert w \Vert_{\frac{r+l}{l}}^{\frac{r+l}{l}}\\
        & \quad +\left( \frac{\xi_{1}c_{1}(r-1)}{r}+\frac{\xi_{2}\lambda_{2}(r-1)}{r}\frac{r}{r+l}\right)\Vert u \Vert_{l+r}^{l+r}.
    \end{aligned}
    \end{equation}

    We now estimate the term $\Vert w \Vert_{\frac{r+l}{l}}$.
    \begin{align*}
        \Vert w \Vert_{\frac{r+l}{l}}&= \Vert (\lambda_{2}I- \Delta)^{-1} f_{2}(u)\Vert_{\frac{r+l}{l}} \leqslant \int_{0}^{\infty} e^{-\lambda_{2}t}\Vert G(\cdot,t) \ast f_{2}(u) \Vert_{\frac{r+l}{l}}dt  \\
        &\leqslant\int_{0}^{\infty} e^{-\lambda_{2}t} \Vert f_{2}(u)\Vert_{\frac{r+l}{l}} dt \leqslant \lambda_{2}c_{2} \left(\int_{\mathbb{R}^{N}} u^{r+l} dx\right)^{\frac{l}{r+l}}
    \end{align*}
    Then we obtain
    \begin{equation*}
        \Vert w \Vert_{\frac{r+l}{l}}^{\frac{r+l}{l}} \leqslant \lambda_{2}c_{2}\Vert u \Vert_{l+r}^{l+r}.
    \end{equation*}
    Next, we estimate the term $\Vert u \Vert_{l+r}^{l+r}$. Using the Lemma \ref{lem11} with $a=2+\frac{2}{N}$, we obtain
    \begin{equation*}
        \Vert u \Vert_{l+r}^{l+r}  \leqslant C(N,l,r)\Vert \nabla u^{\frac{r}{2}}\Vert_{2}^{\frac{2\theta_{4}(l+r)}{r}}\Vert u_{0} \Vert_{1}^{(1-\theta_{4})(l+r)},
    \end{equation*}
    where $\theta_{4}=\frac{r}{2}(1-\frac{1}{l+r})(\frac{1}{N}-\frac{1}{2}+\frac{r}{2})^{-1}$. Since $l<\frac{2}{N}$, we easily check  $\frac{\theta_{4}(l+r)}{r}<1$. Also by Young's inequality, we have
    \begin{equation}\label{22}
       \Vert w \Vert_{\frac{r+l}{l}}^{\frac{r+l}{l}} \leqslant \lambda_{2}c_{2} \Vert u \Vert_{l+r}^{l+r} \leqslant \epsilon_{5} \Vert \nabla u^{\frac{r}{2}}\Vert_{2}^{2} + C(\epsilon_{5},\Vert u_{0} \Vert_{1},\lambda_{2},c_{2},N,l,r).
    \end{equation}
    Similarly
    \begin{equation}\label{23}
        \Vert u \Vert_{r}^{r} \leqslant C(N,r) \Vert \nabla u ^{\frac{r}{2}} \Vert_{2}^{2\theta_{5}}\Vert u_{0}\Vert_{1}^{2(1-\theta_{5})} \leqslant \epsilon_{6}\Vert \nabla u^{\frac{r}{2}}\Vert_{2}^{2} +c(\epsilon_{6},\Vert u _{0}\Vert_{1},r),
    \end{equation}
    where $0< \theta_{5} <1$. Inserting \eqref{22} into \eqref{24}, and then adding \eqref{23}, we obtain
    \begin{align*}
        \frac{1}{r}\frac{d}{dt}&\Vert u \Vert_{r}^{r} + \Vert u \Vert_{r}^{r}\\ &\leqslant   \epsilon_{5} \left(\frac{\xi_{1}c_{1}(r-1)}{\lambda_{2}c_{2}r}+\frac{\xi_{2} \lambda_{2}r(r-1)}{\lambda_{2}c_{2}r(r+l)}+ \frac{l\xi_{2}\lambda_{2}(r-1)}{r(r+l)}\right)\int_{\mathbb{R}^{N}} \vert \nabla u^{\frac{r}{2}}\vert^{2} dx \\  &\quad + \epsilon_{6}\int_{\mathbb{R}^{N}} \vert \nabla u^{\frac{r}{2}}\vert^{2} dx-\frac{4(r-1)}{r^{2}}\int_{\mathbb{R}^{N}} \vert \nabla u^{\frac{r}{2}}\vert^{2} dx +c_{6}.
    \end{align*}
    For sufficiently small $\epsilon_{5},\epsilon_{6}>0$, we see that
    \begin{equation*}
        \frac{1}{r}\frac{d}{dt}\Vert u \Vert_{r}^{r}+\Vert u \Vert_{r}^{r} \leqslant c_{7},
    \end{equation*}
    where $c_{7}$ depends on $l,r,\xi_{1},\xi_{2},N, \lambda_{2},c_{1},c_{2},\Vert u_{0}\Vert_{1}$. Thanks to the comparison principle of the ordinary differential equations, we conclude that
    \begin{equation*}
        \Vert u \Vert_{r} \leqslant c_{8}, \quad 1\leqslant r < \infty,
    \end{equation*}
    where $c_{8}$ depends on $l,r,\xi_{1},\xi_{2},N, \lambda_{2},c_{1},c_{2},\Vert u_{0}\Vert_{1}$. Thus we have \eqref{19}. The same arguments as of \eqref{26} and \eqref{18}, we have \eqref{20} and \eqref{21}.

\end{proof}

\textbf{\emph{Proof of Theorem1.1}}.

~\\
    Using the semigroup property, Lemma \ref{lem9} and \ref{lem10}, we have
    \begin{equation}\label{28}
    \begin{aligned}
        \Vert u \Vert_{\infty}& \leqslant \Vert T(t)u_{0}\Vert_{\infty} +\int_{0}^{t}\Vert T(t-s)\nabla \cdot ( \xi_{1}u(s)\nabla(\Delta-\lambda_{1})^{-1}f_{1}(u(s)))\Vert_{\infty} ds \\
        &\quad+\int_{0}^{t}\Vert T(t-s) \nabla \cdot (\xi_{2}u(s)\nabla(\Delta-\lambda_{2})^{-1}f_{2}(u(s)))\Vert_{\infty} ds \\ & \quad + \int_{0}^{t}\Vert T(t-s)u(s)\Vert_{\infty} ds \\
        &\leqslant e^{-t}\Vert u_{0} \Vert_{\infty} + \xi_{1}\sup_{0<t<T_{max}}(\Vert \nabla v \Vert_{\infty}\Vert u \Vert_{2N})\int_{0}^{t} e^{-(t-s)}(t-s)^{-\frac{3}{4}}ds \\
        &\quad+ \xi_{2}\sup_{0<t<T_{max}}(\Vert \nabla w \Vert_{\infty}\Vert u \Vert_{2N})\int_{0}^{t} e^{-(t-s)}(t-s)^{-\frac{3}{4}}ds \\& \quad +c\sup_{0<t<T_{max}}\Vert u \Vert_{N}\int_{0}^{t} e^{-(t-s)}(t-s)^{-\frac{1}{2}} ds \\
        & \leqslant C( \lambda_{1},\lambda_{2},l,m,N,\xi_{1},\xi_{2},\Vert u_{0}\Vert_{1} ,\Vert u_{0} \Vert_{\infty}).
    \end{aligned}
    \end{equation}
    If $T_{max} < \infty$, then $\Vert u \Vert_{\infty}$ must blow up for finite time. But it is contradiction to \eqref{28}. Therefore the proof of Theorem 1.1 is completed.
\\

\begin{remark}
   By Moser's iteration technique introduced in \cite{moser}, we can also obtain the boundedness of $\Vert u \Vert _{\infty}$  (see
    \cite{regular2,nagai1}).
\end{remark}

\section*{Acknowledgments}

This research was supported by Basic Science Research Program through the National Research Foundation of Korea(NRF) funded by the Ministry of Education (2022R1I1A3055309).

\end{document}